\documentclass[reqno,11pt]{amsart}
\usepackage{amssymb}
\usepackage{amsmath}
\usepackage{mathrsfs}
\usepackage{bm}

\allowdisplaybreaks

\begin{document}

\newtheorem{theorem}{Theorem}
\newtheorem{proposition}{Proposition}
\newtheorem{corollary}{Corollary}
\newtheorem{lemma}{Lemma}

\theoremstyle{definition}
\newtheorem{remark}{Remark}

\hyphenation{Lipschitz}

\sloppy
\title[Analysis of a Living Fluid Continuum Model]
{Analysis of a Living Fluid Continuum Model}

\author[F. Zanger]{Florian Zanger}
\address{Mathematisches Institut, Angewandte Analysis\\
         Heinrich-Heine-Uni\-ver\-sit\"at D\"usseldorf\\
         40204 D\"usseldorf, Germany}
\email{florian.zanger@hhu.de}

\author[H. L\"owen]{Hartmut L\"owen}
\address{Institut f\"ur Theoretische Physik II - Soft Matter\\
         Heinrich-Heine-Uni\-ver\-sit\"at D\"usseldorf\\
	 Universit\"atsstra{\ss}e 1\\
         40225 D\"usseldorf, Germany}
\email{hlowen@thphy.uni-duesseldorf.de}

\author[J. Saal]{J\"urgen Saal}
\address{Mathematisches Institut, Angewandte Analysis\\
         Heinrich-Heine-Uni\-ver\-sit\"at D\"usseldorf\\
         40204 D\"usseldorf, Germany}
\email{juergen.saal@hhu.de}

\date{April 7, 2016}
\thispagestyle{empty}
\parskip0.5ex plus 0.5ex minus 0.5ex
\bibliographystyle{plain}

\begin{abstract}
Generalized Navier-Stokes equations which were proposed recently to describe 
active turbulence in living fluids are analyzed rigorously.
Results on wellposedness and stability in the  
$L^2(\mathbb{R}^n)$-setting are derived. 
Due to the presence of a Swift-Hohenberg term global wellposedness
in a strong setting for arbitrary initial data in $L^2_\sigma(\mathbb{R}^n)$ 
is available. Based on the existence of global strong solutions,
results on linear and nonlinear (in-) stability for the 
disordered steady state and the manifold of ordered polar steady 
states are derived, depending on the involved parameters.
\end{abstract}
\maketitle

\section{Introduction}

There is a need to study analytical properties of generalized Navier-Stokes
equations which were recently proposed
\cite{Wensink-et-al:Meso-scale-turbulence,Dunkel_Heidenreich_PRL_2013,Dunkel_NJP_2013} for
 active soft matter (for recent reviews see
\cite{Romanchuk,Marchetti2013RMP,Menzel,our_RMP})  to
describe the dynamics of ``living fluids'' such as dense bacterial suspensions
at low Reynolds number \cite{Purcell:Life-at-low-Re}.
On a continuum scale, a living fluid flows with an internal speed
that is set by the internal self-propulsion velocity of the bacteria.
Generalized incompressible Navier-Stokes
equations were designed to describe the spontaneous formation of fluid vortices
on the mesoscale by including
higher order derivatives in the velocities entering into the stress tensor.
Indeed, at high density of the bacterial suspension, experiments reveal
the spontaneous formation of meso-scale vortices \cite{Wensink-et-al:Meso-scale-turbulence},
which is confirmed by particle-resolved simulations of
self-propelled particles \cite{Wensink_JPCM_2012} and is consistent with the
predictions of the Navier-Stokes
equations generalized to living fluids. Therefore these continuum equations
constitute an important general framework for flow of living or active fluids
and provide a minimal continuum model for swirling.
Though different to
ordinary turbulence, which occurs at high Reynolds number, this
phenomenon is often called ``active turbulence'' \cite{activeturbulence}.
Active turbulence, which occurs at small Reynolds number,
is characterized by scaling laws different to
ordinary turbulence \cite{Wensink-et-al:Meso-scale-turbulence,Bratanov}.

Therefore a thorough mathematical study of
this generalized Navier-Stokes system is highly desirable both from
a physical and a mathematical point of view.
This paper concerns the  
following minimal hydrodynamic model to describe the bacterial velocity in
the case of highly concentrated bacterial suspensions with negligible density
fluctuations considered on the domain $(0,\infty)\times\mathbb{R}^n$:
\begin{equation}
	\label{eqn:min-hyd-mod}
	\begin{array}{rl}
		v_t+\lambda_0v\cdot\nabla v & =  f-\nabla
		p+\lambda_1\nabla|v|^2-(\alpha+\beta|v|^2)v+\Gamma_0\Delta
		v-\Gamma_2 \Delta^2v\\
		{\mathrm{div\,}}v & =  0\\
		v(0)&=v_0
	\end{array}
\end{equation}
Here $v:(0,\infty)\times\mathbb R^n\to\mathbb{R}^n$ 
is the (vectorial) bacterial velocity field and
$p:(0,\infty)\times\mathbb R^n\to\mathbb{R}$ 
the (scalar) pressure. The first equation is
the conservation of momentum and the equation ${\mathrm{div\,}}v=0$
results from conservation of mass and
the assumption of constant density.

The generalized Navier-Stokes equations defined in  (\ref{eqn:min-hyd-mod}) were
originally proposed by Wensink et al.\ in \cite{Wensink-et-al:Meso-scale-turbulence} 
and then considered in Refs. \cite{Dunkel_Heidenreich_PRL_2013,Dunkel_NJP_2013}.
Clearly, for ${\lambda_0=1}$, $\lambda_1 = \beta = \Gamma_2 = 0$, and $\Gamma_0 >0$, 
the model reduces to the incompressible Navier-Stokes equations in $n$ spatial dimensions.
Let us briefly discuss the physics behind the various terms entering in (\ref{eqn:min-hyd-mod}).
The parameter $\lambda_0$ describes advection and nematic interactions and $\lambda_1$ is 
a prefactor in front of an active pressure contribution \cite{Dunkel_NJP_2013}.
The two parameters $\lambda_0$ and $\lambda_1$ depend on the hydrodynamic nature
of the swimmer, i.e.\ whether they are pushers or pullers \cite{Elgeti}, and
on the dimension $n\in\{2,3\}$.
The term involving the parameters
$\alpha$ and $\beta$  pushes the system towards rest with velocity $v=0$ if $\alpha>0$ and towards a
characteristic non-vanishing velocity of $\sqrt{-\alpha/\beta}$ if $\alpha<0$
and $\beta>0$ as in the Toner-Tu model \cite{Toner-Tu:Long-range} corresponding 
to a quartic Landau-type velocity potential. 
If the parameter $\Gamma_0$ is positive, it determines the
suspension's viscosity similar to the pure Navier-Stokes case. If it is negative, $\Gamma_2>0$
is required for stability reasons as in the traditional Swift-Hohenberg equation
\cite{Swift}, for a review see \cite{Emmerich}. By this fact, here we always
assume $\Gamma_2>0$.

A main objective of this note is to provide an analytical approach to the 
generalized Navier-Stokes equations (\ref{eqn:min-hyd-mod}) in the
$L^2(\mathbb{R}^n)$-setting. This will be performed in Section~\ref{subsec_wp}. 
There we will consider the general system
(\ref{eqn:min-hyd-mod-trans}), which includes the transformed 
systems about the steady states given in
Section~\ref{subsec_steadystates}. These are the {\em disordered
isotropic state} and the manifold of {\em globally
ordered polar states}.
The main results of Section~\ref{subsec_wp} are as follows.
Subsection~\ref{sec_lin} provides an approach to the linearized
equations. The corresponding linear operator admits a bounded
$H^\infty$-calculus, see Lemma~\ref{lem:bounded-Hinf-calc}.
Propositions~\ref{linstabdis} and \ref{linstabord} give precise
information on linear (in-) stability of the steady states depending on
the values of the involved parameters.
In Subsection~\ref{subsec_nlwp} we will prove global (strong) 
wellposedness for the 
generalized Navier-Stokes equations (\ref{eqn:min-hyd-mod}) in 
the $L^2$-setting, see Theorem~\ref{globstrongsol}. 
Based on the global solvability,
Subsection~\ref{subsec_stab} concerns nonlinear (in-) stability.
Theorem~\ref{nonlinstapdis} transfers the linear stability results
for the disordered state to the nonlinear situation.
Theorem~\ref{nonlinstabord} then proves a nonlinear instability
result for the ordered polar state. 

The fact that we can prove the existence of a global unique strong 
solution for arbitrary initial data in $L^2_\sigma(\mathbb{R}^n)$ for system
(\ref{eqn:min-hyd-mod}) of course is due to the presence of the
Swift-Hohenberg term $\Gamma_2\Delta^2 u$. 
It causes the nonlinear terms to appear less strong compared to 
the classical second order Navier-Stokes equations. 
We refer to \cite{galdi2011,sohr2001,wiegner1999,amann2000,cannone2004} 
and the references 
cited therein for more information on the classical Navier-Stokes
equations. 

We also remark that the purpose of this note is not  
to present best possible results in every direction.
The $L^2$-approach given here is merely a first
step towards a thorough analysis of the active fluids continuum model
(\ref{eqn:min-hyd-mod}) in a variety of further significant  situations. 
Further developments and
future projects, e.g.\ including fluid boundaries, 
are addressed in Section~\ref{concl_fut_dev}.

\section{Steady States}\label{subsec_steadystates}

We assume $\Gamma_2,\beta>0$ and $\alpha\in\mathbb{R}$. Then the 
following physically relevant stationary solutions appear:
\begin{equation}\label{disorderedstate}
	(v,p)=(0,p_0)
\end{equation}
with a pressure constant $p_0$ and, if $\alpha<0$, additionally
\begin{equation}\label{orderedstate}
	(v,p)=(V,p_0),
\end{equation}
where $V\in B_{\alpha,\beta}
:=\{x\in \mathbb R^n:\ |V|=\sqrt{-\alpha/\beta}\}$, i.e., $V$ 
denotes a constant vector with arbitrary orientation
and fixed swimming speed $|V|=\sqrt{-\alpha/\beta}$.

The steady state (\ref{disorderedstate}) corresponds to a {\em disordered
isotropic state} and (\ref{orderedstate}) to the manifold
$B_{\alpha,\beta}$ of {\em globally
ordered polar states}.

Note that mathematically there is a further manifold of 
stationary solutions given by
\begin{equation}
	v(x)=v_0,\qquad p(x)=p_0-(\alpha+\beta|v_0|^2)
	v_0\cdot x,\qquad x\in\Omega,\,p_0\in\mathbb{R},
\end{equation}
with $v_0\in\mathbb{R}^3$ arbitrary. For $v_0=0$ or $|v_0|=\sqrt{-\alpha/\beta}$
these solutions correspond to the above steady states (\ref{disorderedstate})
and (\ref{orderedstate}), respectively. For all other values of $v_0$
they are, however, physically not relevant since their pressure
takes arbitrary large negative values. Thus, in the sequel we will only
consider (\ref{disorderedstate}) and (\ref{orderedstate}).

\section{Wellposedness and Stability}\label{subsec_wp}

We perform an approach to the hydrodynamic model (\ref{eqn:min-hyd-mod})
in the $L^2(\mathbb R^n)$-setting. 
In order to include the steady states 
in our analysis we consider the following generalized system:
\begin{equation}
	\label{eqn:min-hyd-mod-trans}
	\begin{array}{r@{\ =\ }l}
		u_t+\lambda_0\left[(u+V)\cdot\nabla\right] u 
		+(M+\beta|u|^2)u-\Gamma_0\Delta
		u+\Gamma_2\Delta^2u+\nabla q& 
		  f+N(u), \\
	{\mathrm{div\,}}u &   0,\\    
	u(0)&u_0.
	\end{array}
\end{equation}
Here $q=p-\lambda_1 |v|^2$, $M\in\mathbb{R}^{n\times n}$ is a symmetric matrix, and
$N(u)=\sum_{j,k}a_{jk}u^ju^k$ with $(a_{jk})_{j,k=1}^n\subset \mathbb{R}^n$ 
is a quadratic nonlinear term. By setting
\begin{equation}\label{valuesds}
	V=0,\quad M=\alpha,\quad N(u)=0
\end{equation}
we obtain (\ref{eqn:min-hyd-mod}) for $u=v$, i.e.,\ the system corresponding to
the steady state (\ref{disorderedstate}) and by setting
\begin{equation}\label{valuesos}
	V\in B_{\alpha,\beta},\quad M=2\beta VV^t,\quad 
	N(u)=-\beta|u|^2V-2\beta(u\cdot V)u
\end{equation}
we obtain the system for $u=v-V$ corresponding to (\ref{orderedstate}).
Note that for the appearing parameters we always assume that 
$\lambda_0,\lambda_1,\Gamma_0,\alpha\in\mathbb{R}$ and that $\Gamma_2,\beta>0$.
Furthermore, dimension is always assumed to be $n\in\{2,3\}$.

For a domain $\Omega\subset \mathbb{R}^n$,
a Banach space $X$, and $1\le p\le \infty$ 
in the sequel $L^p(\Omega,X)$ 
denotes the standard Bochner-Lebesgue space with norm 
\[
	\|u\|_{L^p(X)}=\left(\int_\Omega\|u(x)\|_X^p\,dx\right)^{1/p},
\]
if $1\le p<\infty$ and
$\|u\|_{L^\infty(X)}={\mathrm{ess\,sup}}_{x\in\Omega}\|u\|_X$ if $p=\infty$. 
In case that $\Omega=X=\mathbb{R}^n$ its subspace
of solenoidal functions is denoted by
\[
	L^p_\sigma(\mathbb R^n):=\{v\in L^p(\mathbb R^n);\
	{\mathrm{div\,}}v=0\}.
\]
Note that for $1<p<\infty$ we have the Helmholtz decomposition
\[
	L^p(\mathbb{R}^n)=L^p_\sigma(\mathbb{R}^n)\oplus G_p(\mathbb{R}^n)
\]
with $G_p(\mathbb{R}^n)=\{\nabla p;\ p\in {\mathscr{D}}'(\mathbb{R}^n),\ \nabla p\in L^p(\mathbb{R}^n)\}$.
The associated Helmholtz projector onto $L^p_\sigma(\mathbb{R}^n)$ is represented
as 
\[
	P=\mathcal{F}^{-1}\left(I-\frac{\xi\xi^T}{|\xi|^2}\right)\mathcal{F},
\]
where $\mathcal{F}$ denotes the Fourier transformation and $I$ the identity 
matrix in $\mathbb{R}^n$. 

The symbol $W^{k,p}(\Omega,X)$, $k\in\mathbb{N}_0$, $1\le p\le\infty$, stands
for the standard Sobolev space of $k$-times differentiable functions in
$L^p(\mathbb{R}^n)$. Its norm is given as
\[
	\|f\|_{W^{k,p}(X)}:=\biggl(\sum_{|\alpha|\le k}
	\|\partial^\alpha f\|_{L^p(X)}^p\biggr)^{1/p}
\]
with the usual modification if $p=\infty$. 
The fractional order Sobolev resp.\ Sobolev-Slobodeckij
spaces are defined by complex resp.\ real
interpolation as
\begin{align*}
	W^{k+t,p}(\Omega,X)
	&:=\left[W^{k,p}(\Omega,X),W^{k+1,p}(\Omega,X)\right]_{t},\\
	W^{k+t}_p(\Omega,X)
	&:=\left(W^{k,p}(\Omega,X),W^{k+1,p}(\Omega,X)\right)_{t}
\end{align*}
for $t\in(0,1)$ and $k\in\mathbb{N}_0$. 
For $p=2$ and $s\ge 0$ we use the notation $H^s(\Omega,X):=W^{s,2}(\Omega,X)$
and we frequently write $L^p(\Omega)$, $W^{s,p}(\Omega)$, and 
$W^{s}_p(\Omega)$ if $X=\mathbb{R}^n$. Also note that 
$H^s(\Omega)=W^{s,2}(\Omega)=W^{s}_2(\Omega)$, but that
$W^{s,p}(\Omega,X)\neq
W^{s}_p(\Omega,X)$ in general.
Finally, $\mathscr{L}(X,Y)$ denotes the space of all bounded and 
linear operators from the space $X$ into the space $Y$,
we write $\mathscr{L}(X)$ if $X=Y$, and
$\sigma(A)$ denotes the spectrum of a linear operator
$A:D(A)\subset X\to X$.

\subsection{Linear Theory}\label{sec_lin}

In this subsection we consider the linearized system
\begin{equation}
	\label{lflin}
	\begin{array}{r@{\ =\ }ll}
		u_t+\lambda_0(V\cdot\nabla) u 
		+Mu-\Gamma_0\Delta
		u+\Gamma_2\Delta^2u+\nabla q& 
		  f&\text{in } (0,\infty)\times\mathbb{R}^n, \\
	{\mathrm{div\,}}u &   0&\text{in } (0,\infty)\times\mathbb{R}^n,\\    
	u(0)&u_0&\text{in }\mathbb{R}^n.
	\end{array}
\end{equation}
Thanks to $\Gamma_2>0$ the operator
\[
	A_{_{SH}}u:=\Gamma_2\Delta^2u,\quad 
	u\in D(A_{_{SH}}):=W^{4,p}(\mathbb{R}^n)\cap L^p_\sigma(\mathbb{R}^n),
\]
admits a bounded $H^\infty$-calculus on 
$L^p_\sigma(\mathbb{R}^n)$
with $H^\infty$-angle $\phi_{A_{_{SH}}}^\infty=0$ 
for $p\in(1,\infty)$. This follows as an easy consequence of Mikhlin's
multiplier theorem, for instance. See e.g.\ \cite{denk2003}
for an introduction to the notion of a bounded $H^\infty$-calculus.
Since every other term appearing in (\ref{lflin}), more precisely 
the operator
\begin{equation}\label{eq:operator-B}
	Bu:=\lambda_0(V\cdot\nabla) u
		+PMu-\Gamma_0\Delta u,
\end{equation}
is of lower order, by a standard perturbation argument we 
immediately deduce 
\begin{lemma}\label{lem:bounded-Hinf-calc}
Let $1<p<\infty$. There is an $\omega>0$ such that the operator
$\omega+A_{_{LF}}$, where
\begin{equation}\label{fulllinop}
	A_{_{LF}}u:=(A_{_{SH}}+B) u,\quad 
	u\in D(A_{_{LF}}):=W^{4,p}(\mathbb{R}^n)\cap L^p_\sigma(\mathbb{R}^n),
\end{equation}
admits a bounded $H^\infty$-calculus on 
$L^p_\sigma(\mathbb{R}^n)$
with $H^\infty$-angle $\phi_{\omega+A_{_{LF}}}^\infty<\pi/2$. 
\end{lemma}
As a consequence, cf.\ \cite{denk2003}, $-A_{_{LF}}$ is the generator of an analytic
$C_0$-semigroup on $L^p_\sigma(\mathbb{R}^n)$ and it has maximal regularity:
\begin{corollary}\label{max_reg}
Let $1<p<\infty$, $T\in(0,\infty)$. For $f\in L^p((0,T),L^p_\sigma(\mathbb{R}^n))$
and $u_0\in W^{4-4/p}_p(\mathbb{R}^n) \cap L^p_\sigma(\mathbb{R}^n)$ there exists a unique solution $(u,q)$
of (\ref{lflin}) satisfying
\begin{align*}
	&\|u\|_{W^{1,p}((0,T),L^p)}
	+\|u\|_{L^{p}((0,T),W^{4,p})}
	+\|\nabla q\|_{L^{p}((0,T),L^p)}\\
	&\le C\left(\|f\|_{L^{p}((0,T),L^p)}
	+\|u_0\|_{W^{4-4/p}_p}\right)
\end{align*}
with $C>0$ independent of $u,q,f,u_0$.
\end{corollary}

To obtain preciser information on the spectrum 
of the operator $A_{_{LF}}$ we apply
Fourier transformation to (\ref{fulllinop}) to the result
that
\[
	\sigma_{A_{_{LF}}}(\xi)
	=\mathcal{F} A_{_{LF}}\mathcal{F}^{-1}
	=\Gamma_2|\xi|^4+\Gamma_0 |\xi|^2+\sigma_P(\xi)M
	+i\lambda_0 V\cdot\xi,\quad \xi\in\mathbb{R}^n,
\]
with $\sigma_P(\xi)=(1-\xi\xi^t/|\xi|^2)$ the symbol of the
Helmholtz projector.
We first consider the disordered state (\ref{disorderedstate}). We 
set $A_d:=A_{_{LF}}$ in this case. Then
according to (\ref{valuesds}) the above expression takes the form
\[
	\sigma_{A_d}(\xi)
	=\Gamma_2|\xi|^4+\Gamma_0 |\xi|^2+\alpha,\quad \xi\in\mathbb{R}^n.
\]
Calculating the intersection points of the parabola in $s=|\xi|^2$ 
we obtain
\[
	s_\pm^2=\frac{-\Gamma_0}{\Gamma_2}\left(\frac12\pm
	\sqrt{\frac14-\frac{\alpha\Gamma_2}{\Gamma_0^2}}\right).
\]
Consequently, if $\Gamma_0<0$ there is an unstable band of modes for
$s^2\in(s_-^2,s_+^2)$ provided this interval is nonempty. In this case
the spectral bound
$s(A_d)=\sup\{\mathrm{Re\,} z;\ z\in\sigma(A_d)\}$ of $-A_d$ is positive.
Since it is well known that for analytic $C_0$-semigroups
spectral bound of the generator and growth bound of the 
generated semigroup coincide
\cite{nagel86}, we deduce that $\exp(-tA_d)$ is exponentially
unstable precisely if $\Gamma_0<0$ and $4\alpha<
\Gamma_0^2/\Gamma_2$, 
or if $\Gamma_0\ge 0$ and $\alpha<0$.
This leads to the following result.
\begin{proposition}\label{linstabdis}
Let $\Gamma_2>0$, $\beta>0$, and $1<p<\infty$. If $\Gamma_0<0$, then
the disordered state~(\ref{disorderedstate}) is linearly stable 
if and only if $4\alpha\ge \Gamma_0^2/\Gamma_2$.
If $\Gamma_0\ge 0$, then the disordered state~(\ref{disorderedstate}) is stable 
if and only if $\alpha\ge 0$. To be precise, 
the semigroup $(\exp(-tA_d))_{t\ge 0}$ on $L^p_\sigma(\mathbb{R}^n)$ corresponding to the 
disordered state (\ref{disorderedstate}) is
\begin{enumerate}
\item[(1)]{
exponentially stable if $\Gamma_0<0$ and $4\alpha > \Gamma_0^2/\Gamma_2$,
or if $\Gamma_0\ge 0$ and $\alpha> 0$;}
\item[(2)]{
asymptotically stable if $\Gamma_0<0$ and $4\alpha=
\Gamma_0^2/\Gamma_2$, or if $\Gamma_0\ge 0$ and $\alpha= 0$;}
\item[(3)]{
exponentially unstable if $\Gamma_0<0$ and $4\alpha <
\Gamma_0^2/\Gamma_2$, or if $\Gamma_0\ge 0$ and $\alpha< 0$.}
\end{enumerate}
\end{proposition}
\begin{proof}
It remains to prove (2), the other assertions are obvious by the
discussion above.
On the other hand, in the situation of (2) we see that 
$(\exp(-tA_d))_{t\ge 0}$ is a bounded analytic $C_0$-semigroup
on $L^p_\sigma(\mathbb{R}^n)$, which are known to be asymptotically stable,
see \cite{nagel86}.
\end{proof}

Next, we consider the manifold $B_{\alpha,\beta}$ of ordered polar states
(\ref{orderedstate}). In this case we set $A_o:=A_{_{LF}}$ and the symbol
of this operator according to (\ref{valuesos}) reads as
\[
	\sigma_{A_o}(\xi)
	=\Gamma_2|\xi|^4+\Gamma_0 |\xi|^2+2\beta\sigma_P(\xi)VV^t
	+i\lambda_0 V\cdot\xi,\quad \xi\in\mathbb{R}^n,
\]
with $V\in B_{\alpha,\beta}$. Note that the matrix 
$\sigma_P(\xi)VV^t$ is 
positive semidefinite and that zero is an eigenvalue by the fact
that $V^t x=0$ if $x\in\mathbb{R}^n\setminus\{0\}$ is perpendicular to $V$.
Choosing $x,\xi\in \{V\}^\perp$ such that $|x|=1$ 
and that $|\xi|$ is small enough, we can always achieve that
\[
	x^t\sigma_{A_o}(\xi)x
	=\Gamma_2|\xi|^4+\Gamma_0 |\xi|^2<0,
\]
provided that $\Gamma_0<0$. Thus, here we obtain the following result.
\begin{proposition}\label{linstabord}
Let $1<p<\infty$, $\Gamma_2>0$, $\beta>0$, and $\alpha<0$. The
ordered polar state~(\ref{orderedstate}) is linearly stable
if and only if $\Gamma_0\ge 0$. To be precise, 
the semigroup $(\exp(-tA_o))_{t\ge 0}$ corresponding to the 
ordered state~(\ref{orderedstate}) is
\begin{enumerate}
\item[(1)]{
exponentially unstable on $L^p_\sigma(\mathbb{R}^n)$ if $\Gamma_0<0$;}
\item[(2)]{
asymptotically stable on $L^2_\sigma(\mathbb{R}^n)$ if $\Gamma_0\ge 0$.}
\end{enumerate}
\end{proposition}
\begin{proof}
Assertion (1) is clear.
To see (2) first observe that due to the occurrence of the term
$i\lambda_0 V\cdot\xi$, $(e^{-A_o t})_{t\ge 0}$ is not a bounded
analytic semigroup. Hence we cannot argue as for $A_d$ to deduce
asymptotic stability. Instead, for (2) we restrict ourselves to the case
$p=2$ and proceed as in \cite{hhms2009}: For $v_0\in H^1(\mathbb{R}^n)\cap L^2_\sigma(\mathbb{R}^n)$,
$v(t):=\exp(-tA_o)v_0$ solves
\begin{equation*}
v_t(t)+\Gamma_2\Delta^2v(t)-\Gamma_0\Delta v(t)+\lambda_0(V\cdot\nabla)v(t)+2\beta PVV^tv(t)=0.
\end{equation*}
Multiplication in $L^2(\mathbb{R}^n)$ with $v(t)$ and integration from $t=0$ to $T$ yields
\begin{equation*}
\|v(T)\|_{L^2}^2
+ 2\!\int_0^T\!\!
\left(
  \Gamma_2\|\Delta v(t)\|_{L^2}^2
  +\Gamma_0\|\nabla v(t)\|_{L^2}^2
  +2\beta\|V\cdot v(t)\|_{L^2}^2
\right)
dt
=
\|v_0\|_{L^2}^2
\end{equation*}
for every $T>0$. This has two consequences: First, the semigroup $\exp(-tA_o)$ is
contractive since it is strongly continuous and since
$H^1(\mathbb{R}^n)\cap L^2_\sigma(\mathbb{R}^n)$ is dense in $L^2_\sigma(\mathbb{R}^n)$, and second, 
we have
\begin{equation}\label{intfinite}
\int_0^\infty \|\Delta v(t)\|_{L^2}^2 dt < \infty.
\end{equation}
If $v_0\in H^4(\mathbb{R}^n)$, using the contractiveness we obtain
\begin{align*}
\left|\frac{d}{dt}\|\Delta v(t)\|_{L^2}^2\right|
&=\left|2\langle\Delta v(t),\Delta v_t(t)\rangle\right|\\
&=\left|-2\langle\Delta^2\exp(-tA_o)v_0,A_o v(t)\rangle\right|
  \leq C\|v_0\|_{H^4}^2,
\end{align*}
i.e. $\|\Delta v(\cdot)\|_{L^2}^2\in BC^1(0,\infty)$. Together 
with (\ref{intfinite}) this implies
\begin{equation}\label{laplacetozero}
\lim_{t\to\infty} \|\exp(-tA_o)\Delta v_0\|_{L^2}
= \lim_{t\to\infty} \|\Delta v(t)\|_{L^2} = 0.
\end{equation}
In order to prove asymptotic stability we have to show that for every
$u_0\in L^2_\sigma(\mathbb{R}^n)$, $u(t):=\exp(-tA_o)u_0$ satisfies
\begin{equation*}
\lim_{t\to\infty}\|u(t)\|_{L^2}=0.
\end{equation*}
Considering that $\{\Delta w;\, w\in H^4(\mathbb{R}^n)\cap L^2_\sigma(\mathbb{R}^n)\}$
is dense in $L^2_\sigma(\mathbb{R}^n)$ we can always find
$v_0\in H^4(\mathbb{R}^n)\cap L^2_\sigma(\mathbb{R}^n)$ with $\|u_0-\Delta v_0\|_{L^2}$
arbitrary small. Making once more use of the contractiveness of the
semigroup we obtain
\begin{align*}
\|u(t)\|_{L^2}
&\leq
 \|\exp(-tA_o)u_0-\exp(-tA_o)\Delta v_0\|_{L^2}
 +\|\exp(-tA_o)\Delta v_0\|_{L^2} \\
&\leq
 \|u_0-\Delta v_0\|_{L^2}+\|\exp(-tA_o)\Delta v_0\|_{L^2},
\end{align*}
and (\ref{laplacetozero}) yields the asymptotic stability.
\end{proof}

\subsection{Local and global strong solvability}\label{subsec_nlwp}

We first consider local-in-time wellposedness.
For $T>0$ we define relevant function spaces as
\begin{align*}
  \mathbb{E}_T &:= W^{1,p}((0,T),L^p_\sigma(\mathbb{R}^n)) \cap L^p((0,T),W^{4,p}(\mathbb{R}^n)), \\
  \mathbb{F}_T^1 &:= L^p((0,T),L^p_\sigma(\mathbb{R}^n)), \quad
  \mathbb{F}^2 := W^{4-4/p}_p(\mathbb{R}^n),\\
  \mathbb{F}_T&:=\mathbb{F}^1_T\times\mathbb{F}^2,
\end{align*}
and the linear operator
\[
	L:\mathbb{E}_T\to \mathbb{F}_T,\quad Lu:=(\partial_tu+A_{_{LF}}u,u(0)).
\]
If we also set
\begin{equation}\label{defnonlins}
  H(u) := \beta P|u|^2u+\lambda_0 P(u\cdot\nabla)u - PN(u)
\end{equation}
and
\begin{equation}\label{absfulleq}
	F(u):=Lu+(H(u),0),
\end{equation}
then the full system (\ref{eqn:min-hyd-mod-trans}) is rephrased as
\[
	F(u)=(f,u_0).
\]
\begin{lemma}
Let $p>(4+n)/4$. We have $H\in C^1(\mathbb{E}_T,\mathbb{F}_T)$ and its Fr\'echet derivative 
is represented as
\begin{equation}\label{repfrder}
	DH(v)u=P\sum_{|\alpha|\le 1}b_\alpha\partial^\alpha u
	+\lambda_0P(u\cdot\nabla)v,
	\quad u,v\in\mathbb{E},
\end{equation}
with matrices $b_\alpha=b_\alpha(v)\in
L^\infty((0,T)\times\mathbb{R}^n,\mathbb{R}^{n\times n})$.
\end{lemma}
\begin{proof}
First observe that by \cite[Proposition~1.4.2]{amann95a} we have
\begin{equation}\label{etinlinf0}
	\mathbb{E}_T\,\hookrightarrow\, L^\infty((0,T),W^{4-4/p}_p(\mathbb{R}^n)).
\end{equation}
The assumption $p>(4+n)/4$ in combination with the Sobolev embedding
yields
\begin{equation}\label{etinlinf}
	\mathbb{E}_T\,\hookrightarrow\, L^\infty((0,T)\times\mathbb{R}^n).
\end{equation}
Utilizing this fact we obtain
\begin{align*}
	\|(u\cdot\nabla)u\|_{\mathbb{F}^1_T}
	&\le C\|u\|_{\infty}\|\nabla u\|_{\mathbb{F}^1_T}\le
	C\|u\|_{\mathbb{E}_T}^2,\\
	\||u|^2u\|_{\mathbb{F}^1_T}
	&\le C\||u|^2\|_{\infty}\|u\|_{\mathbb{F}^1_T}\le
	C\|u\|_{\mathbb{E}_T}^3,\\
	\|N(u)\|_{\mathbb{F}^1_T}
	&\le C\|u\|_{\infty}\|u\|_{\mathbb{F}^1_T}\le
	C\|u\|_{\mathbb{E}_T}^2,
\end{align*}
hence $H:\mathbb{E}_T\to \mathbb{F}_T$. By the fact that $H$ consists of bi- and trilinear
terms it is obvious that $H\in C^1(\mathbb{E}_T,\mathbb{F}_T)$ (even $H\in C^\infty(\mathbb{E}_T,\mathbb{F}_T)$). 
The Fr\'echet derivative reads as 
\begin{align*}
  DH(v)u &= \beta P|v|^2 u + 2\beta P (u\cdot v)v
             + \lambda_0 P (u\cdot\nabla)v \\
           & \quad + \lambda_0 P (v\cdot\nabla)u
                    - 2 P \sum_{i,k=1}^n a_{jk} (u^j v^k+u^k v^j).
\end{align*}
From this and (\ref{etinlinf}) representation (\ref{repfrder})
easily follows.
\end{proof}
\begin{remark}
Note that the lower bound $p>(4+n)/4$ is not optimal. But, since 
here we are mainly interested in an $L^2$-approach for dimension
$n=2,3$, it is sufficient for our purposes.
\end{remark}
\begin{lemma}\label{freiso}
Let $p>(4+n)/4$, $T\in(0,\infty)$, and $v\in \mathbb{E}_T$. Then we have
\[
	L+(DH(v),0)\in{{\mathscr L}_{is}}(\mathbb{E}_T,\mathbb{F}_T).
\]
\end{lemma}
\begin{proof}
By employing representation (\ref{repfrder}) for $B(t):=DH(v(t))$ we will show
that $B(\cdot)$ is a lower order perturbation of $L$.
To this end, observe that $p>(4+n)/4$ yields
\[
	W^{3,p}(\mathbb{R}^n)\hookrightarrow L^\infty(\mathbb{R}^n)
\]
and thanks to (\ref{etinlinf0}) also $v\in
L^\infty((0,T),W^{1,p}(\mathbb{R}^n))$. This implies
\begin{align*}
	\|(u\cdot\nabla)v(t)\|_{L^p(\mathbb{R}^n)}
	&\le \|\nabla v(t)\|_{p}\|u\|_\infty\\
	&\le C\|v\|_{L^\infty((0,T),W^{1,p})}
	\frac{1}{\mu^{1/4}}\|(\mu+A_{_{LF}})u\|_p
\end{align*}
for all $u\in D(A_{_{LF}})$ and all $\mu>\mu_0$ with 
$\mu_0>0$ large enough. 
We estimate the second term in (\ref{repfrder}) as
\begin{align*}
	\|\sum_{|\alpha|\le 1}b_\alpha\partial^\alpha u\|_{L^p(\mathbb{R}^n)}
	&\le C\|v\|_{\infty}\|\nabla u\|_p\\
	&\le C\|v\|_{\infty}
	\frac{1}{\mu^{3/4}}\|(\mu+A_{_{LF}})u\|_p
\end{align*}
for all $u\in D(A_{_{LF}})$ and all $\mu>\mu_0$.
This shows that 
\[
	\|B(t)u\|_p\le \frac{C(v)}{\mu^{1/4}}\|(\mu+A_{_{LF}})u\|_p
	\quad (u\in D(A_{_{LF}}),\ \mu>\mu_0).
\]
Thus, choosing $\mu$ large enough and due to
Corollary~\ref{max_reg}, we can apply \cite[Theorem~2.5]{saal2004a}
to the result that
\[
	L+(\mu+DH(v),0)\in{{\mathscr L}_{is}}(\mathbb{E}_T,\mathbb{F}_T).
\]
Since $L+DH(v)$ is linear, we can remove the shift $\mu>0$ 
and the assertion follows.
\end{proof}
Appealing to the local inverse theorem we can now prove the
following result.
\begin{theorem}[local wellposedness]\label{locstrongsol}
Let $\Gamma_2,\beta>0$, $\Gamma_0,\alpha\in\mathbb{R}$,
and $p>(4+n)/4$.
For every $u_0\in W^{4-4/p}_p(\mathbb{R}^n)\cap L^p_\sigma(\mathbb{R}^n)$ and 
$f\in L^p((0,T),L^p_\sigma(\mathbb{R}^n))$ there exists a 
$T>0$ and a
unique solution $(u,q)$ of (\ref{eqn:min-hyd-mod-trans}) such that
\begin{align*}
	u&\in W^{1,p}((0,T),L^p_\sigma(\mathbb{R}^n))
	\cap L^{p}((0,T),W^{4,p}(\mathbb{R}^n)),\\
	\nabla q&\in L^{p}((0,T),L^{p}(\mathbb{R}^n)).
\end{align*}
\end{theorem}
\begin{proof}
We fix $(f,u_0)\in \mathbb{F}_T$ and define a reference solution as
\[
	u^*:=L^{-1}(f,u_0)\in \mathbb{E}_T.
\]
For the Fr\'echet derivative of the nonlinear operator $F\in
C^1(\mathbb{E}_T,\mathbb{F}_T)$ 
given in (\ref{absfulleq}) we obtain in view of Lemma~\ref{freiso}
that
\[
	DF(u^*)=L+DH(u^*)\in {{\mathscr L}_{is}}(\mathbb{E}_T,\mathbb{F}_T).
\]
Hence the local inverse theorem 
yields neighborhoods $U\subset\mathbb{E}_T$ of $u^*$
and $V\subset\mathbb{F}_T$ of $F(u^*)$ such that 
$F:U\to V$ is bijective. 

Now, taking $T'>0$ small enough, we find a solution.
To see this, for $0<T^\prime<T$ we define
$f_{T^\prime}\in\mathbb{F}_T^1$ by
\begin{equation*}
  f_{T^\prime}(t) := \left\{
                       \begin{array}{ll}
                         f(t),         & t\in (0,T^\prime) \\
                         f(t)+H(u^*)(t), & t\in [T^\prime,T).
                       \end{array}
                     \right.
\end{equation*}
The continuity of the integral implies
\begin{equation*}
  f_{T^\prime} \xrightarrow{T^\prime\to 0} f+H(u^*) \text{ in } \mathbb{F}_T^1.
\end{equation*}
Thus, since $F(u^*)=(f+H(u^*),u_0)$, there is $T^\prime>0$ with
$(f_{T^\prime},u_0)\in V$. The unique function $u\in U$ with
$F(u)=(f_{T^\prime},u_0)$ satisfies $F(u)=(f,u_0)$ on $(0,T^\prime)$.
Thus, recovering the pressure via
\[
	\nabla q:=-(I-P)\left[
	\lambda_0\left[u\cdot\nabla\right] u 
		+(M+\beta|u|^2)u-N(u)\right]
		\in L^p((0,T),L^p(\mathbb{R}^n))
\]
we find that $(u,q)|_{(0,T')}$ is a strong solution to
(\ref{eqn:min-hyd-mod-trans}). 
\end{proof}
The just constructed local solution extends to a global one, at least 
if $p=2$.
\begin{theorem}[global wellposedness]\label{globstrongsol}
Let $\Gamma_2,\beta>0$, $\Gamma_0,\alpha\in\mathbb{R}$, $T\in(0,\infty)$. 
For every $u_0\in H^{2}(\mathbb{R}^n)\cap L^2_\sigma(\mathbb{R}^n)$ and 
$f\in L^2((0,T),L^2_\sigma(\mathbb{R}^n))$ there exists a 
unique solution $(u,q)$ of (\ref{eqn:min-hyd-mod-trans}) such that
\begin{align*}
	u&\in H^{1}((0,T),L^2_\sigma(\mathbb{R}^n))
	\cap L^{2}((0,T),H^{4}(\mathbb{R}^n)),\\
	\nabla q&\in L^{2}((0,T),L^2(\mathbb{R}^n)).
\end{align*}
\end{theorem}
\begin{proof}
We derive a priori bounds in the strong class which will give the
result. To this end, we multiply (\ref{eqn:min-hyd-mod-trans})
with $u$ and integrate over $(0,t)\times\mathbb{R}^n$. This yields
\begin{align*}
	&\frac12\|u(t)\|_2^2+\Gamma_2\int_0^t\|\Delta u\|_2^2\,ds
	+\beta\int_0^t\|u\|_4^4\,ds
	= \frac12\|u_0\|_2^2
	+\int_0^t\int_{\mathbb{R}^n}fu\,dxds\\
	&\qquad\strut +\Gamma_0\int_0^t\int_{\mathbb{R}^n}u\Delta u\,dxds
	-\int_0^t\int_{\mathbb{R}^n}uMu\,dxds
	+\int_0^t\int_{\mathbb{R}^n}uN(u)\,dxds
\end{align*}
for $t\in(0,T)$. By applying Cauchy-Schwarz' and Young's inequality
we can estimate as
\[
	\int_0^t\int_{\mathbb{R}^n}u\Delta u\,dxds
	\le \frac{\Gamma_2}{2|\Gamma_0|}\|\Delta u\|_{L^2((0,t),L^2)}^2
	+C\int_0^t\|u\|_2^2\,dt
\]
and 
\begin{equation}\label{estnonglob}
	\int_0^t\int_{\mathbb{R}^n}uN(u)\,dxds
	\le \frac{\beta}{2}\|u\|_{L^4((0,t),L^4)}^4
	+C\int_0^t\|u\|_2^2\,dt.
\end{equation}
Plugging this into the above equality we arrive at
\begin{align*}
	&\|u(t)\|_2^2+\|u\|_{L^2((0,t),H^2)}^2
	+\|u\|_{L^4((0,t),L^4)}^4\\
	&\quad \le C\left(\|u_0\|_2^2
	+\|f\|_{L^2((0,t),L^2)}^2\right)
	+C\int_0^t\|u(s)\|_2^2\,ds\quad (t\in(0,T)).
\end{align*}
Hence, Gronwall's lemma yields
\begin{equation}\label{weakuniest}
\begin{split}
	&\|u\|_{L^\infty((0,T),L^2)}^2+\|u\|_{L^2((0,T),H^2)}^2
	+\|u\|_{L^4((0,T),L^4)}^4\\
	&\quad\le C(1+Te^{\omega T})\left(\|u_0\|_2^2
	+\|f\|_{L^2((0,T),L^2)}^2\right)
\end{split}
\end{equation}
for some $C,\omega>0$. 
	
Next we multiply (\ref{eqn:min-hyd-mod-trans}) with $-\Delta u$
and obtain by utilizing integration by parts
\begin{equation}\label{secstepglob}
\begin{split}
	&\frac12\|\nabla u(t)\|_2^2+\Gamma_2\int_0^t\|\nabla\Delta
	u\|_2^2\,ds
	+\beta\int_0^t\int_{\mathbb{R}^n}(\nabla|u|^2u)\nabla u\,dxds\\
	&= \frac12\|\nabla u_0\|_2^2
	-\int_0^t\int_{\mathbb{R}^n}f\Delta u\,dxds
	-\Gamma_0\|\Delta u\|_{L^2((0,t),L^2)}^2\\
	&\qquad\strut -\int_0^t\int_{\mathbb{R}^n}(\nabla u)M\nabla u\,dxds
	-\int_0^t\int_{\mathbb{R}^n}(\Delta u)N(u)\,dxds\\
	&\qquad\strut -\lambda_0\int_0^t\int_{\mathbb{R}^n}[(V+u)\cdot\nabla]
	u\Delta u\,dxds.
\end{split}
\end{equation}
Concerning the third term on the left hand side we calculate
\begin{align*}
	(\nabla|u|^2u)\nabla u
	&=\sum_{j,k,\ell=1}^n(\partial_ku^j)
	\partial_k(u^\ell)^2u^j\\
	&=\sum_{j,k,\ell=1}^n
	(u^\ell)^2(\partial_ku^j)^2
	+2\sum_{k=1}^n(u\cdot\partial_ku)^2.
\end{align*}
This shows that this term is non-negative, hence it drops out.
For the fifth term on the right hand side analogously to (\ref{estnonglob})
we obtain 
\[
	\int_0^t\int_{\mathbb{R}^n}(\Delta u)N(u)\,dxds
	\le C\left(\|u\|_{L^4((0,t),L^4)}^4
	+\int_0^t\|\Delta u\|_2^2\,ds\right),
\]
whereas the last term on the right hand side can be estimated 
utilizing integration by parts and ${\mathrm{div\,}}(V+u)=0$ as 
\begin{align*}
	&\int_0^t\int_{\mathbb{R}^n}[(V+u)\cdot\nabla]u\Delta u\,dxds\\
	&\le C\left(\|u\|_{L^2((0,t),L^2)}^2+\|u\|_{L^4((0,t),L^4)}^4
	\right)
	+\frac{\Gamma_2}{2|\lambda_0|}\int_0^t\|\nabla \Delta u\|_2^2\,dt.
\end{align*}
Plugging this into (\ref{secstepglob}) yields in combination with
(\ref{weakuniest}) that
\begin{equation}\label{weakuniest2}
\begin{split}
	&\|u\|_{L^\infty((0,T),H^1)}^2+\|u\|_{L^2((0,T),H^3)}^2
	+\|u\|_{L^4((0,T),L^4)}^4\\
	&\quad\le C(1+Te^{\omega T})\left(\|u_0\|_{H^1}^2
	+\|f\|_{L^2((0,T),L^2)}^2\right).
\end{split}
\end{equation}

In the third step we multiply with $\Delta^2 u$ to obtain
\begin{equation}\label{lap2est}
\begin{split}
	&\frac12\|\Delta u(t)\|_2^2+\Gamma_2\int_0^t\|\Delta^2
	u\|_2^2\,ds\\
	&= \frac12\|\Delta u_0\|_2^2
	+\int_0^t\int_{\mathbb{R}^n}f\Delta^2 u\,dxds
	-\Gamma_0\|\Delta\nabla u\|_{L^2((0,t),L^2)}^2\\
	&\qquad\strut -\int_0^t\int_{\mathbb{R}^n}(\Delta u)M\Delta u\,dxds
	+\int_0^t\int_{\mathbb{R}^n}(\Delta^2 u)N(u)\,dxds\\
	&\qquad\strut -\lambda_0\int_0^t\int_{\mathbb{R}^n}[(V+u)\cdot\nabla]
	u\Delta^2 u\,dxds
	-\beta\int_0^t\int_{\mathbb{R}^n}|u|^2u\Delta^2 u\,dxds.
\end{split}
\end{equation}
It is enough to focus on the last two terms on the right hand side.
The remaining terms can be handled very similar as before.
The first one of those two terms can be controlled 
by utilizing the estimate
\begin{align*}
	\int_0^t\|(u\cdot\nabla)u\|_2^2\,ds
	&\le C\int_0^t\|u\|_4^2\|\nabla u\|_4^2\,ds\\
	&\le C\|u\|_{L^\infty(H^1)}^2\int_0^t\|u\|_{H^2}^2\,ds\\
	&\le C(1+Te^{\omega T})^2\left(\|u_0\|_{H^1}^2
	+\|f\|_{L^2((0,T),L^2)}^2\right)^2,
\end{align*}
which is valid thanks to (\ref{weakuniest2}) 
and the Sobolev embedding $H^1(\mathbb{R}^n)\hookrightarrow L^4(\mathbb{R}^n)$. 

For the last term, using complex interpolation \cite{triebel} we obtain
\[
	L^{\infty}((0,T),H^1(\mathbb R^n))\cap
	L^2((0,T),H^3(\mathbb R^n))
	\hookrightarrow L^{2/s}((0,T),H^{2s+1}(\mathbb{R}^n))
\]
for $s\in[0,1]$. Taking into account $n\le 3$ the Sobolev embedding
yields
\begin{equation}\label{wssinl4}
	H^{5/3}\hookrightarrow L^6(\mathbb{R}^n).
\end{equation}
Hence, by setting $s=1/3$ we obtain
\[
	L^{\infty}((0,T),H^1(\mathbb R^n))\cap
	L^2((0,T),H^3(\mathbb R^n))
	\hookrightarrow L^{6}((0,T),L^6(\mathbb{R}^n)).
\]
On the other hand, from Cauchy-Schwarz' and Young's inequality we see that
\[
	|\int_0^t\int_{\mathbb{R}^n}|u|^2u\Delta^2 u\,dxds|
	\le C\|u\|_{L^6((0,t),L^6)}^6+\frac{\Gamma_2}{2\beta}
	\|\Delta^2 u\|_{L^2((0,t),L^2)}^2.
\]
Here the second term on the right hand side is absorbed 
by the left hand side of (\ref{lap2est}) 
and the first term is again controlled by estimate (\ref{weakuniest2}).
Summarizing, we arrive at
\begin{align*}
	&\|u\|_{L^\infty((0,T),H^2)}^2+\|u\|_{L^2((0,T),H^4)}^2\\
	& \quad\le C(1+Te^{\omega T})^3\left(\|u_0\|_{H^1}^2
	+\|f\|_{L^2((0,T),L^2)}^2\right)^3.
\end{align*}
Using equations (\ref{eqn:min-hyd-mod-trans}) it is straight forward
to derive similar bounds for the quantities $\|u\|_{H^1((0,T),L^2)}$
and $\|\nabla q\|_{L^2((0,T),L^2)}$ as well.
Thus the assertion is proved.
\end{proof}

\subsection{Nonlinear (In-) Stability}\label{subsec_stab}

Most of the outcome on linear (in-) stability in the $L^2$-setting 
transfers to the corresponding
nonlinear situation. For the transfer of the stability results
we apply energy methods 
and for the transfer of the results on instability 
we employ Henry's instability theorem \cite[Corollary 5.1.6]{henry1981}, 
which we reformulate suitably for our purposes.
\begin{proposition}
\label{thm_henry}
Let $-A$ be the generator of a holomorphic $C_0$-semigroup in a Banach
space $X$ and let $f\colon U\to X$, where $U$ is an open neighborhood in
$X^\gamma:=D(A^\gamma)$ for some $\gamma\in(0,1)$, be locally Lipschitz. Let $x_0\in
D(A)\cap U$ be an equilibrium point of 
\begin{equation}
\label{eq_henry}
\dot w(t)+Aw(t)=f(w(t)),
\end{equation}
i.e. $Ax_0=f(x_0)$. Suppose
\begin{align*}
f(x_0+z)&=f(x_0)+Bz+g(z),\quad g(0)=0,\\
\|g(z)\|&=O(\|z\|_{X^\gamma}^s),
\quad\text{as }z\to0\text{ in }X^\gamma,
\end{align*}
for some $s>1$, $B\in\mathscr{L}(X^\gamma,X)$, and
	$\sigma(-A+B)\cap\{z\in\mathbb{C}:\ \mathrm{Re\,} z>0\}\neq\emptyset$. 
Then $x_0$ is nonlinearly unstable in the following sense: 
there is a constant $\varepsilon_0>0$ such that for any $\delta>0$ there
exists $x\in X^\gamma$ with $\|x-x_0\|_\gamma<\delta$ such that there is
some finite time $t_0>0$ with
\begin{align*}
\|w(t_0,x)\|_{X^\gamma}\geq\varepsilon_0,
\end{align*}
where $w(\cdot,x)$ denotes the solution of (\ref{eq_henry}) with
initial value $w(0,x)=x$.
\end{proposition}
For instability also the following lemma will be helpful. 
\begin{lemma}\label{nonlinestinl2}
Let the nonlinearity $H$ be given as in (\ref{defnonlins}). Then we have
$H\in C^1(H^\sigma(\mathbb{R}^n),L^2_\sigma(\mathbb{R}^n))$ for every $\sigma\ge 5/4$
and 
\[
	\|H(u)\|_2\le C\|u\|_{H^\sigma}^2\quad (\|u\|_{H^\sigma}\le 1).
\]
\end{lemma}
\begin{proof}
Employing H\"older's inequality we obtain
\[
	\|(u\cdot\nabla)u\|_2\le \|u\|_p\|\nabla u\|_q
\]
for $1/p+1/q=1/2$. It is easily checked that the Sobolev embeddings 
$H^\gamma(\mathbb{R}^n)\hookrightarrow L^p(\mathbb{R}^n)$ and $H^{\gamma-1}(\mathbb{R}^n)\hookrightarrow L^q(\mathbb{R}^n)$ 
are sharp for $p=12$, $q=12/5$, and $\gamma=5/4$.
By the fact that
\[
	\||u|^2u\|_2=\|u\|_6^3\le C\|u\|_{H^1}^3
\]
and since $H$ consists of bi- and trilinear terms the assertion follows.
\end{proof}
\begin{remark}
It is clear that due to better Sobolev embeddings the lower bound 
on $\sigma$ can be improved if $n=2$.
\end{remark}
As before, first we consider the disordered state
(\ref{disorderedstate}).
\begin{theorem}\label{nonlinstapdis}
Let $\Gamma_2>0$ and $\beta>0$. Then the 
disordered state (\ref{disorderedstate}) is nonlinearly
\begin{enumerate}
\item[(1)]{
(globally) exponentially stable in $L^2_\sigma(\mathbb{R}^n)$ 
if $\Gamma_0<0$ and $4\alpha > \Gamma_0^2/\Gamma_2$,
or if $\Gamma_0\ge 0$ and $\alpha> 0$;}
\item[(2)]{
stable in $L^2_\sigma(\mathbb{R}^n)$ if $\Gamma_0<0$ and $4\alpha=
\Gamma_0^2/\Gamma_2$, or if $\Gamma_0\ge 0$ and $\alpha= 0$;}
\item[(3)]{
unstable in $H^{\gamma}(\mathbb{R}^n)\cap L^2_\sigma(\mathbb{R}^n)$ for
$\gamma\in[5/16,1)$ if $\Gamma_0<0$ and $4\alpha <
\Gamma_0^2/\Gamma_2$, or if $\Gamma_0\ge 0$ and $\alpha< 0$.}
\end{enumerate}
\end{theorem}
\begin{proof}
Suppose $u$ and $q$ with regularity as in Proposition~\ref{globstrongsol}
solve the nonlinear system (\ref{eqn:min-hyd-mod-trans}) corresponding to the
disordered state~(\ref{disorderedstate}), that is
\begin{equation*}
u_t+\Gamma_2\Delta^2 u-\Gamma_0\Delta u+\lambda_0(u\cdot\nabla)u
+(\alpha+\beta|u|^2)u+\nabla q=0.
\end{equation*}
Testing this equation with $u$ we obtain
\begin{equation*}
\frac{1}{2}\frac{d}{dt}\|u\|_2^2
+\Gamma_2\|\Delta u\|_2^2
+\Gamma_0\|\nabla u\|_2^2
+\alpha\|u\|_2^2
+\beta\|u\|_4^4=0.
\end{equation*}
If $\Gamma_0\geq 0$ and $\alpha\geq 0$, all coefficients are nonnegative and
we deduce
\begin{equation*}
\frac{d}{dt}\|u\|_2^2 \leq -2\alpha\|u\|_2^2,
\end{equation*}
which yields
\begin{equation*}
\|u(t)\|_2^2\leq e^{-2\alpha t}\|u_0\|_2^2\qquad (t\geq 0),
\end{equation*}
i.e.\ stability if $\alpha=0$ and exponential stability if $\alpha>0$.

If $\Gamma_0<0$, we use the Plancherel theorem, H\"older's inequality,
and Young's inequality with $\varepsilon$ to estimate the term
\begin{align*}
\|\nabla u\|_2^2
& = \int_{\mathbb{R}^n}|\xi|^2|\hat{u}(\xi)|^2d\xi
    \leq \||\xi|^2\hat{u}(\xi)\|_2 \|\hat{u}\|_2 \\
& = \|\Delta u\|_2 \|u\|_2
    \leq \frac{\varepsilon^2}{2}\|\Delta u\|_2^2 + \frac{1}{2\varepsilon^2}\|u\|_2^2.
\end{align*}
The $\varepsilon^2/2$-term can be absorbed by the $\Gamma_2$-term if we
choose $\varepsilon^2=2\Gamma_2/|\Gamma_0|$. Dropping the $\beta$-term as well
we are left with
\begin{equation*}
\frac{1}{2}\frac{d}{dt}\|u\|_2^2
+\alpha\|u\|_2^2 \leq \frac{\Gamma_0^2}{4\Gamma_2}\|u\|_2^2.
\end{equation*}
As before, this implies stability if $4\alpha=\Gamma_0^2/\Gamma_2$ and
exponential stability if $4\alpha>\Gamma_0^2/\Gamma_2$.
Thus assertions (1) and (2) are proved.

To see instability we apply Proposition~\ref{thm_henry}. In our
situation we have $x_0=0$, $B=0$, $x=u$, $A=A_d$, and
$f(u)=H(u)$. 
From this we also see that $H(0)=0$, i.e., that $g=f=H$ in Henry's
notation. Note that Proposition~\ref{linstabdis}(3) under the 
assumption (3) above implies that
$\sigma(-A_d)\cap \{z\in\mathbb{C}:\ \mathrm{Re\,} z>0\}\neq \emptyset$.
Next, the fact that $\omega+A_d$ admits a bounded 
$H^\infty$-calculus for some $\omega>0$ (Lemma~\ref{lem:bounded-Hinf-calc})
yields
\[
	D(A_d^\gamma)=\left[L^2_\sigma(\mathbb{R}^n),D(A_d)\right]_\gamma
	=H^{4\gamma}(\mathbb{R}^n)\quad (\gamma\in[0,1]),
\]
where $[\cdot,\cdot]_s$ denotes the complex interpolation 
space, cf.\ \cite{triebel}. 
Lemma~\ref{nonlinestinl2} then implies that the assertions of 
Proposition~\ref{thm_henry} are fulfilled 
for $\gamma\in[5/16,1)$ and $s=2$.
Hence the disordered state is unstable. 
\end{proof}
Finally we consider instability for the ordered polar state.
\begin{theorem}\label{nonlinstabord}
Let $\Gamma_2,\beta>0$ and $\Gamma_0, \alpha<0$. Then the
ordered polar state (\ref{orderedstate}) is nonlinearly unstable
in $H^{\gamma}(\mathbb{R}^n)\cap L^2_\sigma(\mathbb{R}^n)$ for $\gamma\in[5/16,1)$. 
\end{theorem}
\begin{proof}
Based on Lemma~\ref{nonlinestinl2} and Proposition~\ref{linstabord}(1)
the proof is analogous to the proof of Theorem~\ref{nonlinstapdis}(3).
\end{proof}
\begin{remark}
We have seen in Theorem~\ref{nonlinstapdis} that for the disordered
steady state the results on linear
(in-) stability in principle completely transfer to the nonlinear situation. 
Note that for the time being it is not clear if for $\Gamma_0\ge 0$
the linear stability for the ordered polar state 
given by Proposition~\ref{linstabord}(2) transfers
to the nonlinear situation as well. 
Proceeding as for the disordered state, i.e., employing energy methods,
it appears that for the ordered state 'disturbing' terms on the right hand side
can not easily be absorbed by the 'good' terms on the left hand side by
just applying Young, Sobolev and H\"older. It seems that here 
a refined analysis is required which, however, is left as a future
challenge.
\end{remark}

\section{Conclusions and future developments}\label{concl_fut_dev}

In conclusion, we have shown that the model proposed by Wensink et al.\
in \cite{Wensink-et-al:Meso-scale-turbulence} gives rise to a
mathematically wellposed system which reflects the asymptotic
behavior observed in simulations and experiments.
In detail we have proved: 
\begin{enumerate}
\item[(i)] existence of a unique local-in-time solution for
initial data in $W^{4-4/p}_p\cap L^p_\sigma(\mathbb{R}^n)$;
\item[(ii)] existence of a unique global strong solution for
arbitrary initial data in $H^2(\mathbb{R}^n)\cap L^2_\sigma(\mathbb{R}^n)$;
\item[(iii)] results on stability and instability of the ordered
and the disordered steady states in the $L^2$-setting 
depending on the values of the occurring physically relevant parameters.
\end{enumerate}
Note that (ii) is in contrast to the (mathematical) situation of the 
classical Navier-Stokes system. The fact that we
can prove existence of a unique global strong solution here, of course
essentially relies on the presence of the fourth order term 
in (\ref{eqn:min-hyd-mod}) which
provides sufficient regularity.

The intention of this note is to give an analytical approach in 
the $L^2$-setting which serves as a first step for 
further thorough examinations in several directions. 
Future work should address the following problems: first of all,
it should also be mentioned that the generalized Navier-Stokes
framework was augmented even further by including further higher derivatives
in the velocities entering into the stress tensor,
see the recent work by Slomka and Dunkel
\cite{Slomka_EPJST_2015}. We anticipate that our analysis can
also be performed in this more
general case provided the highest order term has the correct stabilizing sign.

Second, another more complex problem is that of boundary conditions
for the fluid velocity field. This is important to take into
account walls and obstacles which confine the bacterial flow.
In fact, recent experiments with bacterial turbulence were
performed with mobile wedge-like
obstacles or carriers \cite{Kaiser1,Kaiser2,Kaiser3}
which were powered by activity
or in static meso-structured environments \cite{Wioland_2016} where
there is an interesting competition between the geometric structure
of the boundaries and the swirling. Recent observations have also addressed
spheres as passive additives to steer bacterial turbulence. For all these
interesting set-ups
the mathematical analysis should include boundary conditions.
The latter are typically stick or slip or involve a finite slipping length.
Therefore the  wellposedness and stability of the generalized
Navier-Stokes equation
in nontrivial boundaries  should be addressed in future studies.

Third, some bacteria move on the surface of an emulsion droplet,
which has motivated recent studies of active particles
on a compact manifold, such as a sphere
\cite{Henkes_2015,Grossmann,Li}. It would be interesting to
generalize the hydrodynamic model (\ref{eqn:min-hyd-mod}) 
towards a
nonplanar geometry and to prove
wellposedness and stability for the problem on curved space.


\bibliography{living_fluids}
\bibliographystyle{plain}

\end{document}